\documentclass[12pt]{amsart}
\usepackage{amscd,amssymb,hyperref}
\usepackage{cite}
\let\OLDthebibliography\thebibliography
\renewcommand\thebibliography[1]{
  \OLDthebibliography{#1}
  \setlength{\parskip}{1pt}
  \setlength{\itemsep}{1pt plus 0.3ex}
}
\usepackage{amsfonts}
\usepackage[all]{xy}
\usepackage{xcolor}
\usepackage{graphicx, epsfig}
\newtheorem{theorem}{Theorem}[section]
\newtheorem{cor}[theorem]{Corollary}
\newtheorem{lemma}[theorem]{Lemma}
\newtheorem{prop}[theorem]{Proposition}

\theoremstyle{definition}

\theoremstyle{remark}

\numberwithin{equation}{subsection}
\theoremstyle{plain}

\newtheorem{problem}[theorem]{Problem}

\def\square{\vbox{
      \hrule height 0.4pt
      \hbox{\vrule width 0.4pt height 5.5pt \kern 5.5pt \vrule width 0.4pt}
      \hrule height 0.4pt}}

\def\ch\mathrm{c h}

\long\def\symbolfootnote[#1]#2{\begingroup%
\def\thefootnote{\fnsymbol{footnote}}\footnote[#1]{#2}\endgroup}

\numberwithin{equation}{section}

\begin{document}
\title{On embeddings of quandles into groups}

\author{Valeriy Bardakov,~Timur Nasybullov}
\date{}
\maketitle
\begin{abstract}
In the present paper, we introduce the new construction of quandles. For a group $G$ and its subset $A$ we construct a quandle $Q(G,A)$ which is called the $(G,A)$-quandle and study properties of this quandle. In particular, we prove that if $Q$ is a quandle such that the natural map $Q\to G_Q$ from $Q$ to its enveloping group $G_Q$ is injective, then $Q$ is the $(G,A)$-quandle for an appropriate group $G$ and its subset $A$.
Also we introduce the free product of quandles and study this construction for $(G,A)$-quandles. In addition, we classify all finite quandles with enveloping group $\mathbb{Z}^2$.

~\\
\noindent\emph{Keywords: quandle, enveloping group, free product.} \\
~\\
\noindent\emph{Mathematics Subject Classification: 20N02,  57M27.}
\end{abstract}
\section{Introduction}

 A quandle is an algebraic system whose axioms are derived from the Reidemeister moves on oriented link diagrams. Such algebraic systems were introduced independently by Joyce \cite{Joy} and Matveev \cite{Mat} as an invariant for knots in $\mathbb{R}^3$. More precisely,  to each oriented diagram $D_K$ of an oriented knot $K$ in $\mathbb{R}^3$ one can associate the quandle $Q(K)$ which does not change if we apply the Reidemeister moves to the diagram $D_K$.  Moreover, Joyce and Matveev proved that two knot quandles $Q(K_1)$ and $Q(K_2)$ are isomorphic if and only if $K_1$ and $K_2$ are weakly equivalent, i.~e. there exists a homeomorphism of $\mathbb{R}^3$ (possibly, orientation reversing) which maps $K_1$ to $K_2$.  Over the years, quandles have been investigated by various authors in order to construct new
invariants for knots and links (see, for example, \cite{Carter, FenRou, Kamada, NanSinSin, Nelson}).

A lot of quandles can be constructed from groups. The most common example of a quandle is the conjugation quandle ${\rm Conj}(G)$ of a group $G$, i.~e. the quandle $\langle G,*\rangle$ with $x*y=y^{-1}xy$ for $x,y\in G$. If we define another operation $*$ on the group $G$, namely $x*y=yx^{-1}y$, then the set $G$ with this operation also forms a quandle. This quandle is called the core quandle of a group $G$ and is denoted by ${\rm Core}(G)$. In particular, if $G$ is an abelian group, then the quandle ${\rm Core}(G)$ is called the Takasaki quandle of the abelian group $G$ and is denoted by $T(G)$. Such quandles were studied by Takasaki in \cite{Takasaki}. If $G$ is the cyclic group of order $n$, then the Takasaki quandle $T(G)$ is called the dihedral quandle of order $n$ and is denoted by ${\rm R}_n$. Other examples of quandles arising from groups can be found, for example, in \cite{BarDeySin, BarNasSin}. Algebraic properties of quandles were studied in \cite{BarDeySin, BarNasSin, BarPasSin, BarSinSin, ElhFerTsv, ElhSaiZap}.

We say that a qandle $Q$ is embedded into a group, if there is a group $G$ such that $Q \leq {\rm Conj}(G)$.  If $Q \leq {\rm Conj}(G)$, then there exists a subgroup $H\leq G$ such that $Q$ is a union of conjugacy classes of $H$. As was noticed in \cite[Section 4]{Joy}, not every quandle is embedded into a group. The question of describing quandles which are embedded into groups was formulated in \cite[Question 3.1]{BarDeySin}.

In the present paper, we introduce a new way of constructing quandles from groups.
For a group $G$ and its subset $A$ we construct the quandle $Q(G,A)$ which is called the $(G,A)$-quandle and study properties of this quandle. In particular, we prove that if the natural map $Q\to G_Q$ from a quandle $Q$ to its enveloping group $G_Q$ is injective, then there exists a subset $A$ in $G_Q$ such that $Q=Q(G_Q,A)$ (Theorem~\ref{quasiuniv}). As a consequence, we prove that commutative quandles, latin quandles, simple quandles and Takasaki quandles are $(G,A)$-quandles (Corollary~\ref{simlatcom} and Theorem~\ref{fintak}). Also we introduce the notion of the free product $Q*P$ of quandles $Q$, $P$. We prove that if $Q,P$ are quandles such that the natural maps $Q\to G_Q$, $P\to G_P$ are injective, and $A,B$ are sets of representatives of orbits of $Q$, $P$ under the action of inner automorphisms, then $Q*P=Q(G_Q*G_P,A\cup B)$, where $G_Q*G_P$ denotes the free product of the groups $G_Q$, $G_P$ (Theorem~\ref{freere}).

The paper is organized as follows. In Section~\ref{secfr}, we give preliminaries about quandles. In Section~\ref{U1mR}, we find explicit presentations of enveloping groups of some quandles. In Section~\ref{GAprop}, we introduce the $(G,A)$-quandles and study some proerties of these quandles. In Section~\ref{latinica}, we prove that commutative, latin and simple quandles are $(G,A)$-quandles. In Section~\ref{taktaktak}, we prove that Takasaki quandles are $(G,A)$-quandles. Finally, in Section~\ref{freeqi}, we introduce the notion of the free product of quandles and study the free products of $(G,A)$-quandles.
\section{Preliminary definitions and examples}\label{secfr}
A rack $R$ is an algebraic system with one binary algebraic operation $(x,y)\mapsto x*y$ which satisfies the following two axioms:
\begin{itemize}
\item[(r1)] the map $S_x:y\mapsto y*x$ is a bijection of $R$ for all $x\in R$,
\item[(r2)] $(x*{y})*z=(x*z)*({y*z})$ for all $x,y,z\in R$.
\end{itemize}
Axioms (r1) and (r2) imply that the map $S_x$ for $x\in R$ is an automorphism of $R$. The group ${\rm Inn}(R)=\langle S_x~|~x\in R\rangle$  generated by all $S_x$ for $x\in R$ is called  the group of inner automorphisms of $R$. The group ${\rm Inn}(R)$ acts on $R$, the orbit of an element $x\in R$ under this action is denoted by ${\rm Orb}(x)$ and is called the orbit of $x$. If $R={\rm Orb}(x)$ for some $x\in R$, then $R$ is said to be connected. For the sake of simplicity for elements $x,y\in R$ we denote by $y*^{-1}x$ the element $S_x^{-1}(y)$, and sometimes we denote by $y*^1x=y*x$. Generally, the operation $*$ is not associative. For elements
$x_1,\dots,x_n\in R$ and integers $\varepsilon_2,\dots,\varepsilon_n\in\{\pm1\}$ we denote by
$$x_1*^{\varepsilon_2}x_2*^{\varepsilon_3}x_3*^{\varepsilon_4}\dots *^{\varepsilon_n}x_n=((\dots((x_1*^{\varepsilon_2}x_2)*^{\varepsilon_3}x_3)*^{\varepsilon_4}\dots)*^{\varepsilon_n}x_n).$$
	A rack $R$ which satisfies an additional axiom
\begin{itemize}
\item[(q1)] $x*x=x$ for all $x\in R$
\end{itemize}
is called a quandle. The simplest example of a quandle is the trivial quandle on a set $X$, that is the quandle $Q=\langle X,*\rangle$, where $x*y=x$ for all $x,y\in X$. If $|X|=n$, then the trivial quandle on $X$ is denoted by $T_n$.

A free quandle on a nonempty set $X$ is a quandle  $FQ(X)$ together with a map $\varphi : X \to FQ(X)$  such that for any other map $\rho : X \to Q$, where $Q$ is a quandle, there exists a
unique homomorphism $\overline{\rho} : FQ(X) \to Q$ such that $\overline{\rho}\circ\varphi = \rho$. The free quandle is unique up to isomorphism. The definition of the free rack $FR(X)$ is the same. We use the following construction of the  free rack $FR(X)$ and the free quandle $FQ(X)$ on the set $X$ of generators (see \cite{FenRou} and \cite[Section~8.6]{Kam2}).

Let $F(X)$ be the free group with the free generators $X$. On the set $X\times F(X)$ define the operation
$$ (a, u) * (b, v) = (a, u v^{-1} b v)$$
for $a, b \in X$, $u, v \in F(X)$. The algebraic system $FR(X) =(X\times F(X), *) $ is the free rack on $X$. For $(a, u), (b, v)\in FR(X)$ the element $(a, u) *^{-1} (b, v)$ is given by
$$(a, u) *^{-1} (b, v) = (a, u v^{-1} b^{-1} v).$$
The free quandle $FQ(X)$ on $X$ can be constructed as the quotient of $FR(X)$ modulo the equivalence relation $(a,w)\sim(a,aw)$ for $a\in X$, $w\in F(X)$. The construction of the free quandle $FQ(X)$ as a subquandle of  the conjugation quandle ${\rm Conj}(F(X))$ of the free group $F(X)$ is introduced in \cite[Theorem 4.1]{Joy}. If $X=\{x_1,\dots,x_n\}$ is a finite set, then we denote the free quandle $FQ(X)$ by $FQ_n$.

A lot of interesting examples of quandles come from groups. If $G$ is a group and $x,y\in G$, then denote by $x*y=y^{-1}xy$. The algebraic system $\langle G,*\rangle$ is a quandle which is called the conjugation quandle of $G$ and is denoted by ${\rm Conj}(G)$. The map ${\rm Conj}$ which maps $G$ to ${\rm Conj}(G)$ is a functor from the category of groups to the category of quandles.

For a quandle $Q$ denote by $G_Q$ the group with the set of generators $Q$ and the set of relations $x*y=y^{-1}xy$ for all $x,y\in Q$. The group $G_Q$ is called the enveloping group of the quandle $Q$ \cite[Section 6]{Joy} (see also \cite{BarNasSin}). For a quandle $Q$ there exists a homomorphism $Q\to T_n$ from $Q$ to the trivial quandle with $n$ (number of orbits in $Q$) elements. This homomorphism sends all elements from ${\rm Orb}(x)$ to $x$. This homomorphism induces the surjective homomorphism $G_Q\to G_{T_n}=\mathbb{Z}^n$, therefore $G_Q$ is always  infinite.

Joyce \cite{Joy} denoted the group $G_Q$ by ${\rm Adconj}(Q)$. The map ${\rm  Adconj}$ which maps $Q$ to $G_Q$ is a functor from the category of quandles to the category of groups which is adjoint to the functor ${\rm Conj}$. There exists a natural map $\theta:Q\to G_Q$ which maps an element $x\in Q$ to the corresponding generator of $G_Q$. This map $\theta$ satisfies the following universal property: for any quandle homomorphism $f : Q \to {\rm Conj}(G)$ for a group $G$, there exists a unique group  homomorphism $F : G_Q \to G$ such that $f = F \circ \theta$.

\section{Enveloping groups of certain quandles}\label{U1mR}
For a group $G$ denote by ${\rm Conj}_{-1}(G)$ the quandle $\langle G,*\rangle$ with the operation given by $x*y=yxy^{-1}$ for $x,y\in G$. This quandle is obtained from the quandle ${\rm Conj}(G)$ changing the operations $*$ and $*^{-1}$. The following way of constructing a new quandle from given quandles $Q_1, Q_2$ is introduced in \cite[Proposition 9.2]{BarNasSin}.
\begin{prop}\label{oldunion}
Let $\langle Q_1, *\rangle$, $\langle Q_2, \circ\rangle$ be quandles and $\sigma: Q_1 \to  {\rm Conj}_{-1} \left({\rm Aut}(Q_2) \right)$, $\tau: Q_2 \to  {\rm Conj}_{-1} \left({\rm Aut}(Q_1) \right)$ be quandle homomorphisms. Then the set $Q=Q_1 \sqcup Q_2$ with the operation
$$
x\star y=\begin{cases}
x*y,& x, y \in Q_1, \\
x\circ y,  &x, y \in Q_2, \\
{\tau(y)}(x),  &x \in Q_1, y \in Q_2, \\
{\sigma(y)}(x) &x \in Q_2, y \in Q_1.
\end{cases}
$$
is a quandle if and only if the following conditions hold:
\begin{enumerate}
\item $\tau(z)(x)* y=\tau\left(\sigma(y)(z)\right)(x* y)$ for $x, y \in Q_1$ and $z \in Q_2$,
\item $\sigma(z)(x)\circ y=\sigma\left(\tau(y)(z)\right)(x\circ y)$ for $x, y \in Q_2$ and $z \in Q_1$.
\end{enumerate}
\end{prop}
\noindent The quandle $\langle Q_1\sqcup Q_2,\star\rangle$ introduced in Proposition~\ref{oldunion} is denoted by $Q=Q_1\underset{\sigma,\tau}{\sqcup} Q_2$ and is called the union of quandles $Q_1,Q_2$ (with respect to $\sigma,\tau$).

A subquandle $P$ of a quandle $Q$ is said to be normal if
$p * q$ belongs to $P$ for all  $p\in P,q\in Q$. It is clear that if $P$ is a normal subquandle of $Q$, then $P$ is a union of orbits in $Q$. Quandles $Q_1, Q_2$ are obviously normal in $Q=Q_1\underset{\sigma,\tau}{\sqcup} Q_2$, therefore if $Q$ is  a connected quandle, then it cannot be presented as a union of two subquandles. If $Q$ is not connected, $Q_1$ is an orbit of some element from $Q$, $Q_2=Q\setminus Q_1$, then $Q=Q_1\underset{\sigma,\tau}{\sqcup} Q_2$ for appropriate $\sigma,\tau$. So, a quandle $Q$ can be presented as a union of two quandles if and only if $Q$ is not connected. The dihedral quandle ${\rm R}_{2n}$ of an even order can be presented as a union of two subquandles in a very nice way.
\begin{prop}\label{dihdecom} ${\rm R}_{2n}={\rm R}_n\underset{\sigma,\tau}{\sqcup} {\rm R}_n$ for appropriate $\sigma,\tau$.
\end{prop}
\begin{proof} The quandle ${\rm R}_{2n}$ consist of the elements $x_0,\dots,x_{2n-1}$ and it is a union of two orbits ${\rm Orb}(x_0)=\{x_{2m}~|~m=0,\dots,n-1\}$, ${\rm Orb}(x_1)=\{x_{2m+1}~|~m=0,\dots,n-1\}$. The formulas
\begin{align}
\notag x_{2m+1}*x_{2k+1}=x_{2(2k-m)+1},&&x_{2m}*x_{2k}=x_{2(2k-m)}
\end{align}
show that both ${\rm Orb}(x_0)$, ${\rm Orb}(x_1)$ are isomorphic to ${\rm R}_n$.
\end{proof}

\subsection{Quandles with the free abelian enveloping group}\label{exenvexen}
Let $n,m$ be positive integers and $T_n=\{x_0,\dots,x_{n-1}\}$, $T_m=\{y_0,\dots,y_{m-1}\}$ be trivial quandles. Denote by $\sigma$ the map which send every element from $T_n$ to the cyclic permutation $f=(y_0, y_1,\dots, y_{m-1})$, and by $\tau$ the map which send every element from $T_m$ to the cyclic permutation $g=(x_0, x_1,\dots, x_{n-1})$. It is easy to check that the maps $\sigma,\tau$ satisfy Proposition \ref{oldunion}. Therefore we can define the union $U(n,m)=T_n\underset{\sigma,\tau}{\sqcup} T_m$ of $T_n, T_m$ with respect to $\sigma,\tau$. The operation in $U(n,m)=\{x_0,\dots,x_{n-1},y_0,\dots,y_{m-1}\}$ is given by the formulas
\begin{align}
\notag x_i*x_j&=x_i,&& x_i*y_r=x_{i+1~({\rm mod}~n)},\\
\label{envel} y_r*y_s&=y_r, &&y_r*x_i=y_{r+1~({\rm mod}~m)}.
\end{align}
It is clear that $U(1,1)=T_2$. The quandle  $U(1,2)$ is isomorphic to the quandle
$$Cs(4)=\langle x,y,z~|~x*y=z, x*z=x, z*x=z, z*y=x, y*x=y, y*z=y\rangle$$
introduced by Joyce in \cite[Section 6]{Joy}.

\begin{theorem}\label{abenvel} Let $Q$ be a finite quandle. Then $G_Q=\mathbb{Z}\times\mathbb{Z}$ if and only if $Q$ is isomorphic to $U(n,m)$ for coprime integers $n,m$.
\end{theorem}
\begin{proof}From formulas (\ref{envel}) we conclude that $G_{U(n,m)}$ is generated by the elements $x_0,\dots,x_{n-1},y_0,\dots,y_{m-1}$ and has the defining relations
\begin{align}
\label{comxy2} [x_i,x_j]&=[y_r,y_s]=1&&\text{for}~i,j=0,\dots, n-1,~r,s=0,\dots,m-1\\
\label{YXY1} y_r^{-1}x_iy_r&=x_{i+1~({\rm mod}~n)}&&\text{for}~i=0,\dots, n-1,~r=0,\dots,m-1\\
\label{XYX2} x_i^{-1}y_rx_i&=y_{r+1~({\rm mod}~m)}&&\text{for}~i=0,\dots, n-1,~r=0,\dots,m-1
\end{align}
From formulas (\ref{YXY1}), (\ref{XYX2}) we see that
\begin{align}
\label{xiii}x_i=y_0^{-i}x_0y_0^i,~y_r=x_0^{-r}y_0x_0^r&& \text{for}~i=0,\dots, n-1, r=0,\dots,m-1,
\end{align}
therefore we can delete the the elements $x_1,\dots,x_{n-1},y_1,\dots,y_{m-1}$ from the generating set of $G_{U(n,m)}$. Using formulas (\ref{xiii}) relations  (\ref{comxy2}) can be rewritten in the form
\begin{align}
\notag [x_0,x_0^{y_0^j}]=[y_0,y_0^{x_0^{s}}]=1&&\text{for}~j=0,\dots, n-1,~s=0,\dots,m-1.
\end{align}
These relations are equivalent to only two relations $[x_0,[x_0,y_0]]=[y_0,[x_0,y_0]]=1$ which say that the element $[x_0,y_0]$ belongs to the center of $G_{U(n,m)}$. This fact and formulas (\ref{xiii}) imply that
\begin{align}
\label{yiii}x_i=x_0[x_0,y_0]^i,~y_r=y_0[x_0,y_0]^r&& \text{for}~i=0,\dots, n-1, r=0,\dots,m-1,
\end{align}
Using this fact we conclude that $G_{U(n,m)}$ is generated by the elements $x_0,y_0$ and has the following defining relations (which follow from (\ref{comxy2}), (\ref{YXY1}), (\ref{XYX2})).
\begin{align}
\label{GUnm} [x_0,[x_0,y_0]]=[y_0,[x_0,y_0]]=[x_0,y_0]^n=[x_0,y_0]^m=1.
\end{align}
If $n,m$ are coprime, relations (\ref{GUnm}) imply that $G_{U(n,m)}=\mathbb{Z}\times\mathbb{Z}$. The if-part of the theorem is proved.

Let $Q$ be a quandle such that $G_Q=\mathbb{Z}\times\mathbb{Z}$. Decompose $Q$ into the union of the orbits $Q={\rm Orb}(x_1)\sqcup \dots \sqcup{\rm Orb}(x_k)$ and consider the homomorphism $Q\to T_k$ which maps each element from ${\rm Orb}(x_k)$ to the same element (which is denoted by the same letter $x_k$). From the definition of the enveloping group follows that
$$\mathbb{Z}^2=G_Q/[G_Q,G_Q]=G_T=\mathbb{Z}^k,$$
therefore $k=2$ and $Q$ is a union of only two orbits $Q={\rm Orb}(x)\sqcup {\rm Orb}(y)$. Therefore $Q$ can be presented in the form
\begin{equation}\label{decompose}
Q={\rm Orb}(x)\underset{\sigma,\tau}{\sqcup} {\rm Orb}(y)
\end{equation}
for appropriate $\sigma,\tau$. Consider now the homomorphism $\varphi:Q\to {\rm Conj}(G_Q)$ which maps every element from $Q$ to the corresponding generator of $G_Q$. Since $G_Q$ is abelian, $\varphi({\rm Orb}(x))=\{\varphi(x)\}$, $\varphi({\rm Orb}(y))=\{\varphi(y)\}$. Since ${\rm Inn}(Q)$ is the quotient of $G_Q$,  for all $a\in Q$, $z\in {\rm Orb}(x)$, $t\in{\rm Orb}(y)$  we have
\begin{align}
\label{trivactions} a*z=a*x,&& a*t=a*y,
\end{align}
From these equalities follows that for all $a,b\in {\rm Orb}(x)$ we have
$$a*b=a*x=a*a=a,$$
i.~e. ${\rm Orb}(x)=T_n$ is a trivial quandle. Similarly we can show that ${\rm Orb}(y)=T_m$ is a trivial quandle. So, equality (\ref{decompose}) can be rewritten in the following form
\begin{equation}\label{decompose2}
Q=T_n\underset{\sigma,\tau}{\sqcup} T_m
\end{equation}
for appropriate $\sigma,\tau$. From equalities (\ref{trivactions}) follows that for $a\in {\rm Orb}(x)=T_n$ and  $b\in {\rm Orb}(y)=T_m$ we have
$$a*b=a*y$$
but by formula (\ref{decompose2}) and the definition of the union of quandles (Proposition \ref{oldunion}) we have
$$\tau(b)(a)=a*b=a*y=\tau(y)(a),$$
therefore $\tau(b)=\tau(y)=g$ for all $b\in {\rm Orb}(y)$ and some $g\in {\rm Aut}({\rm Orb}(x))={\rm Aut}(T_n)$. Similarly we can prove that $\sigma$ is a constant map which maps each element from  ${\rm Orb}(x)$ to some automorphism $f\in {\rm Aut}({\rm Orb}(y))={\rm Aut}(T_m)$. Since ${\rm Orb}(x)=T_n$ forms an orbit in $Q$ and the only elements from $Q$ which act nontrivially on  $T_n$ belong to ${\rm Orb}(y)$ (these elements act on ${\rm Orb}(x)=T_n$ as $g$), the group generated by $g$ must act transitively on $T_n$, i.~e. $g$ is a cycle of length $n$. Similarly we can prove that $f$ is a cycle of length $m$. Now formula (\ref{decompose2}) implies that $Q=U(n,m)$. From formulas (\ref{GUnm}) and the fact that $G_Q=\mathbb{Z}\times\mathbb{Z}$ follows that $m,n$ are coprime.
\end{proof}
\begin{cor}
There is an infinite number of finite pairwise non-isomorphic quandles with the same enveloping group.
\end{cor}
The following problem comes naturally from Theorem \ref{abenvel}.
\begin{problem}
Classify quandles with free abelian enveloping groups.
\end{problem}

If $n,m$ are not coprime, then from formulas (\ref{GUnm}) follows that $G_{U(n,m)}$ is nilpotent of nilpotency class $2$. Similarly to $U(n,m)$ we define the quandle
$$U(\infty,\infty)=T_{\infty}\underset{\sigma,\tau}{\sqcup} T_{\infty}$$
which is a union of two trivial quandles $\{x_i~|~i\in \mathbb{Z}\}$ and $\{y_j~|~j\in \mathbb{Z}\}$ with the operation given by
\begin{align}
\notag x_i*x_j&=x_i,&& x_i*y_r=x_{i+1},\\
\notag y_r*y_s&=y_r, &&y_r*x_i=y_{r+1}.
\end{align}
From the proof of Theorem \ref{abenvel} we conclude that the enveloping group $G_{U(\infty,\infty)}$ is the free nilpotent group of rank $2$ and class $2$.
\begin{problem}
Classify quandles with nilpotent enveloping groups.
\end{problem}

\subsection{Dihedral quandles} Let $n$ be a positive integer and ${\rm R}_n$ be the dihedral quandle of order $n$. This quandle consists of the elements $x_0,\dots,x_{n-1}$ with the operation given by $x_i*x_j=x_{2j-i~({\rm mod}~n)}$. Throughout this section we denote by $a_i$ the image of the element $x_i$ under the natural map $Q\to G_Q$.

\begin{prop}\label{2krep}Let $n$ be a positive integer. Then the enveloping group $G_{{\rm R}_{2n}}$ of the quandle ${\rm R}_{2n}$ is generated by the elements $a_0,a_1$ and is defined by the relations
\begin{equation}\label{defrel}
[a_0,a_1^2]=[a_0^2,a_1]=[a_0,(a_0a_1)^{n}]=[a_1,(a_0a_1)^{n}]=1.
\end{equation}
For $\varepsilon\in\{0,1\}$ and $r\in\{0,\dots,n-1\}$ the element $a_{2r+\varepsilon}$ from $G_{{\rm R}_{2n}}$ has the presentation $a_{2r+\varepsilon}=a_{\varepsilon}^{(a_0a_1)^r}$.
\end{prop}
\begin{proof} Since in ${\rm R}_{2n}$ we have the equality
$$(x_0*x_1)*x_1=x_2*x_1=x_0,$$
in $G_{{\rm R}_{2n}}$ there is an equality $a_1^{-2}a_0a_1^2=a_0$, or equivalently $[a_0,a_1^2]=1$. Similarly we can show that $[a_0^2,a_1]=1$.

 Using direct calculations it is easy to see that in ${\rm R}_{2n}$ we have the identity
 $$
\notag x_{0}=x_{0}\underbrace{*x_0*x_1*x_0*x_1*\dots*x_0*x_1}_{2n~\text{symbols}~*}.
$$
therefore in $G_{{\rm R}_{2n}}$ there is an identity $a_0^{(a_0a_1)^{n}}=a_0$, or equivalently $[a_0,(a_0a_1)^{n}]=1$. Similarly we can show that the relation $[a_1,(a_0a_1)^{n}]=1$ holds in $G_{{\rm R}_{2n}}$. Therefore relation (\ref{defrel}) are true in $G_{{\rm R}_{2n}}$. Now we need to prove that $G_{{\rm R}_{2n}}$ is generated by only two elements $a_0$, $a_1$ and that all the relation of $G_{{\rm R}_{2n}}$ follow from relation (\ref{defrel}).

By the definition, the group $G_{{\rm R}_{2n}}$ has the generators $a_0,\dots,a_{2n-1}$ and is defined by the relations
\begin{equation}\label{justrel}
a_j^{-1}a_ia_j=a_{2j-i ({\rm mod}~2n)}~\text{for}~i,j\in\{0,\dots,2n-1\}.
\end{equation}
 Using direct calculations it is easy to check that for $\varepsilon\in\{0,1\}$ and $r\in\{0,\dots,n-1\}$ there is an identity
\begin{align}
\notag x_{2r+\varepsilon}&=x_{\varepsilon}*\underbrace{x_0*x_1*x_0*x_1*\dots*x_0*x_1}_{2r~\text{symbols}~*}.
\end{align}
Therefore in $G_{{\rm R}_{2n}}$ there is an identity
\begin{equation}\label{allel}
a_{2r+\varepsilon}=a_{\varepsilon}^{(a_0a_1)^r}
\end{equation}
and $G_{{\rm R}_{2n}}$ is generated by $a_0$, $a_1$. Let us prove that every relation from (\ref{justrel}) follows from relations (\ref{defrel}). Let $i=2r+\varepsilon$, $j=2s+\eta$ for $r,s\in\{0,\dots,n-1\}$, $\varepsilon,\eta\in\{0,1\}$. Using equalities (\ref{defrel}) checking all the $4$ possibilities for the pair $(\varepsilon, \eta)$ it is easy to check that the equality
\begin{equation}\label{ezjob}
a_{\varepsilon}^{a_{\eta}(a_0a_1)^{\varepsilon-\eta}}=a_{\varepsilon}
\end{equation}
always holds. Then we have
\begin{align}
\notag a_j^{-1}a_ia_j&=a_{2s+\eta}^{-1}a_{2r+\varepsilon}a_{2s+\eta}&&\\
\notag &=\left(a_{\eta}^{(a_0a_1)^s}\right)^{-1}\left(a_{\varepsilon}^{(a_0a_1)^r}\right)\left(a_{\eta}^{(a_0a_1)^s}\right)&&\text{using}~(\ref{allel})\\
\notag &=a_{\varepsilon}^{(a_0a_1)^r(a_0a_1)^{-s}a_{\eta}(a_0a_1)^s}&&\\
\notag &=a_{\varepsilon}^{a_{\eta}(a_0a_1)^{2s-r}}&&\text{using}~(\ref{defrel})\\
\notag &=a_{\varepsilon}^{a_{\eta}(a_0a_1)^{\varepsilon-\eta}(a_0a_1)^{2s-r+\eta-\varepsilon}}&&\\
\notag &=a_{\varepsilon}^{(a_0a_1)^{2s-r+\eta-\varepsilon}}&&\text{using}~(\ref{ezjob})\\
\notag &=a_{\varepsilon}^{(a_0a_1)^{2s-r+\eta-\varepsilon({\rm mod}~n)}}&&\text{using}~(\ref{defrel})\\
\notag &=a_{2j-i({\rm mod}~2n)}&&\text{using}~(\ref{allel})
\end{align}
So, we proved that every relation from (\ref{justrel}) follows from relations (\ref{defrel}).
\end{proof}
\begin{cor}\label{dihinj} Let $n$ be a positive integer and $Q={\rm R}_{2n}$ be the dihedral quandle of order $2n$. Then the natural map $Q\to G_Q$ is injective.
\end{cor}
\begin{proof} From Proposition \ref{2krep} follows that the elements $a_0^2$, $a_1^2$ belong to the center of $G_{{\rm R}_{2n}}$. Therefore $H=\langle a_0^2,a_1^2\rangle=\mathbb{Z}\times\mathbb{Z}$ is a normal subgroup of $G_{{\rm R}_{2n}}$. From relations~(\ref{defrel}) follows that the quotient $G_{{\rm R}_{2n}}/H$ is generated by the elements $y_0=a_0H$, $y_1=a_1H$ and has the defining relations $y_0^2=y_1^2=(y_0y_1)^{2n}=1$.
This group is isomorphic to the dihedral group $D_{2n}$ and we have the short exact sequence
\begin{equation}\label{sesd}1\to \mathbb{Z}\times \mathbb{Z}\to G_{{\rm R}_{2n}}\to D_{2n}\to 1.
\end{equation}
Denote by $\varphi:G_{{\rm R}_{2n}}\to G_{{\rm R}_{2n}}/H=D_{2n}$ the canonical homomorphism. This homomorphism maps $a_0$ to $y_0$, and $a_1$ to $y_1$. From Proposition \ref{2krep} for $\varepsilon\in\{0,1\}$ and $r\in\{0,\dots,n-1\}$ the element $a_{2r+\varepsilon}$ from $G_{{\rm R}_{2n}}$ has the presentation $a_{2r+\varepsilon}=a_{\varepsilon}^{(a_0a_1)^r}$. Therefore
\begin{equation}\label{uniex}\varphi(a_{2r+\varepsilon})=y_{\varepsilon}(y_0y_1)^{2r}=y_{0}(y_0y_1)^{2r+\varepsilon}.
\end{equation}
Since $D_{2n}=\mathbb{Z}_{2n}\rtimes\mathbb{Z}_2=\langle y_0y_1\rangle\rtimes\langle y_0\rangle$, every element from $D_{2n}$ can be uniquely expressed in the form $y_0^{\eta}(y_0y_1)^{s}$ for $\eta\in\{0,1\}$ and $s\in\{0,\dots,2n-1\}$. This fact together with equalities (\ref{uniex}) implies that the elements $\varphi(a_{i})$ are different for different~$i$. Therefore all the elements $a_0,\dots,a_{2n-1}$ are different.
\end{proof}
From the short exact sequence (\ref{sesd}) follows that the quotient of $G_{{\rm R}_{2n}}$ by the central subgroup is isomorphic to $D_{2n}$. If $n=2^k$ for some integer $k$, then $D_{2n}$ is nilpotent, therefore $G_{{\rm R}_{2n}}$ is nilpotent, so, the dihedral quandles of order $2^k$ provide another examples of quandles with nilpotent enveloping group.
\section{$(G,A)$-quandles and their properties}\label{GAprop}
Based on the construction of the free quandle introduced in Section~\ref{secfr} we will introduce the new way of constructing quandles from groups. Let $G$ be a group and $A$ be a subset of $G$. On the set $A\times G$ define the operation
$$
(a, u) * (b, v) = (a, u v^{-1} b v)
$$
for $a, b \in A$, $u, v \in G$. The algebraic system $R(G,A) =(A\times G, *) $ is a rack and for $a, b \in A$, $u, v \in G$ we have
$$(a, u) *^{-1} (b, v) = (a, u v^{-1} b^{-1} v).$$
We call the rack $R(G,A)$ the $(G,A)$-rack. Denote by $Q(G,A)$ the quotient of $R(G,A)$ by the equivalence relation $(a,gh)\sim(a,h)$ for $g\in C_G(a)=\{x\in G~|~xa=ax\}$. The operation in $Q(G,A)$ is given by the formula
$$
[(a, u)] * [(b, v)] = [(a, u v^{-1} b v)],
$$
 where $[(a,u)]$ denotes the equivalence class of $(a,u)$. Is is easy to check that this operation is correct (i.~e. it does not depend on the representative of $[(a,u)]$). Also it is clear that  $Q(G,A)$ is a quandle. We call this quandle the $(G,A)$-quandle. For the sake of simplicity for $a\in A$, $u\in G$ we will write $(a,u)$ instead of $[(a,u)]$, and for $a\in A$ we denote by the same symbol $a$ the element $(a,1)\in Q(G,A)$. If $F(X)$ is the free group freely generated by the set $X$, then $Q(F(X),X)=FQ(X)$.

Joyce \cite[Section 9]{Joy} introduced the notion of the augmented quandle. An augmented quandle $(Q,G)$ consists of a set $Q$, a group $G$ equipped with a right action on the set $Q$, and a function $\varepsilon: Q\to G$ called the
augmentation map which satisfies the following two axioms
\begin{itemize}
\item[(AQ1)] $q\cdot\varepsilon(q)=q$ for all $q\in Q$.
\item[(AQ2)] $\varepsilon(q\cdot x)=x^{-1}\varepsilon(q)x$ for all $q\in Q$, $x\in G$.
\end{itemize}
Given a pair $(Q,G)$ satisfying (AQ1) and (AQ2), one can define on $Q$ the following operations
$$p*q=p\cdot\varepsilon(q),~~~~~~p*^{-1}q=p\cdot \varepsilon(q)^{-1}.$$
The algebraic system $\langle Q,*\rangle$ (together with the action of $G$ and the map $\varepsilon$) is called the augmented quandle. The action of $G$ on $Q$ is by quandle automorphisms, and the augmentation map $\varepsilon: Q\to G$ gives a quandle homomorphism $Q\to {\rm Conj}(G)$.

Let $G$ be a group and $A$ be a nonempty subset of $G$. On the set $A\times G$ define an equivalence relation $(a,g)\sim (a,hg)$ for $h\in C_g(a)$ and denote by $Q=A\times G/\sim$ the set of equivalence classes. The group $G$ acts on $Q$ by the rule $(a,g)\cdot h=(a,gh)$. Define the map $\varepsilon:Q\to G$ by the rule $\varepsilon(a,g)=g^{-1}ag$, this map is correctly defined on the set of equivalence classes. It is easy to check that $\varepsilon$ satisfied (AQ1) and (AQ2) and that the $(G,A)$-quandle is the augmented quandle with acting group $G$. As a corollary we have the following lemma.
\begin{lemma}\label{simple}For an element $g\in G$ the map $\iota_g:Q(G,A)\to Q(G,A)$ which maps $(a,u)$ to $(a,ug)$ is an automorphism of $Q(G,A)$.
\end{lemma}
The following lemma follows from the definition of the $(G,A)$-quandle.
\begin{lemma}\label{simple2} Let $x,y\in Q(G,A)$ be such that $x*y=x$, then $y*x=y$.
\end{lemma}
\begin{proof}
Direct calculations.
\end{proof}
 At first glance it may seem that as for free quandles the quandle $Q(G,A)$ is isomorphic to the subquandle of ${\rm Conj}(G)$ which consists of the union of conjugacy classes of elements from $A$, however, it is not true. For example, if $G=F_2=\langle x,y\rangle$ is a free group and $A=\{x,y,y^{-1}xy\}$, then the quandle $Q(G,A)$ has three orbits under the action of ${\rm Inn}(Q(G,A))$, but the quandle $Q$ which is a subquandle of ${\rm Conj}(G)$ which consists of the union $x^G\cup y^{G}\cup (y^{-1}xy)^G=x^G\cup y^G$ has only two orbits under the action of ${\rm Inn}(Q)$, therefore these quandles are not isomorphic. However, if the elements of $A$ are pairwise non-conjugate, then we have the following proposition.
\begin{prop}Let $G$ be a group and $A$ be a subset of $G$ such that the elements of $A$ are pairwise non-conjugate in $G$. Then  $Q(G,A)$ is isomorphic to the subquandle of ${\rm Conj}(G)$ which consists of the union of conjugacy classes of elements from $A$.
\end{prop}
\begin{proof} Denote by $Q$ the subquandle of ${\rm Conj}(G)$ which consists of the union of conjugacy classes of elements from $A$ and consider the map $\varphi:Q(G,A)\to Q$ given by the rule $(a,u)\varphi=u^{-1}au$. If $a$ is an element of $Q$, then there exist elements $b\in A$, $x\in G$ such that $a=x^{-1}bx$ and $(b,x)\varphi=a$, therefore $\varphi$ is surjective. If $u^{-1}au=(a,u)\varphi=(b,v)\varphi=v^{-1}bv$ for $a,b\in A$, $u,v\in G$, then since the elements from $A$ are pairwise non-conjugate, $a=b$ and $vu^{-1}\in C_G(a)$, therefore $(a,u)=(b,v)$ and $\varphi$ is injective. Using direct calculation we have
\begin{align}
\notag ((a,u)*(b,v))\varphi&=(a,uv^{-1}bv)\varphi=(uv^{-1}bv)^{-1}a(uv^{-1}bv)\\
\notag &=(v^{-1}bv)^{-1}(u^{-1}au)(v^{-1}bv)=\big((b,v)\varphi\big)^{-1}\big((a,u)\varphi\big)\big((b,v)\varphi\big),
\end{align}
therefore $\varphi$ is an isomorphism.
\end{proof}

\begin{cor} Let $G$ be a group and $A$ be the set of representatives of conjugacy classes of $G$. Then  $Q(G,A)$ is isomorphic to ${\rm Conj}(G)$.
\end{cor}

A lot of quandles can be constructed as $(G,A)$-quandles.

\begin{theorem}\label{quasiuniv}Let $Q$ be a quandle such that the natural map $\theta:Q\to G_Q$ is injective, and $A$ be a set of representatives of orbits of $Q$. Then $Q$ is isomorphic to $Q(G_Q,\theta(A))$.
\end{theorem}
\begin{proof} Since $\theta$ is injective, we will identify elements from $Q$ and their images in $G_Q$.

Since $A$ is a maximal subset of $Q$ such that every two elements from $A$ belong to different orbits, for every element $a\in Q$, there exists an element $x\in A$ and an inner automorphism $\varphi\in {\rm Inn}(Q)$  such that $a=x\varphi$. Since $\varphi$ is an inner automorphism of $Q$, it can be written in the form $\varphi=S_{x_1}^{\varepsilon_1}S_{x_2}^{\varepsilon_2}\dots S_{x_n}^{\varepsilon_n}$ for $x_1,\dots,x_n\in Q$, $\varepsilon_1,\dots,\varepsilon_n\in\{\pm1\}$, therefore every element $a$ from $Q$ can be written in the form
\begin{align}\label{canon}a=x*^{\varepsilon_1}x_1*^{\varepsilon_2}\dots *^{\varepsilon_n}x_n
\end{align}
for $x\in A$, $x_1,\dots,x_n\in Q$, $\varepsilon_1, \varepsilon_2,\dots,\varepsilon_n\in\{\pm1\}$. Let $a$ be an element from $Q$ written as in (\ref{canon}) and let $\psi:Q\to Q(G_Q,\theta(A))$ be the map defined by
$$\psi:a\mapsto (x,x_1^{\varepsilon_1}\dots x_n^{\varepsilon_n}).$$
We are going to prove that $\psi$ gives an isomorphism $Q\to Q(G_Q,\theta(A))$. At first, we need to prove that $\psi$ is well defined, i.~e. if $a=b$, then $\psi(a)=\psi(b)$. Let
\begin{align}
\notag a&=x*^{\varepsilon_1}x_1*^{\varepsilon_2}\dots *^{\varepsilon_n}x_n,\\
\label{korre} b&=y*^{\eta_1}y_1*^{\eta_2}\dots *^{\eta_k}y_k
\end{align}
for $x,y\in A$, $x_1,\dots,x_n,y_1\dots,y_k\in Q$, $\varepsilon_1,\dots,\varepsilon_n,\eta_1,\dots,\eta_k\in\{\pm 1\}$ and $a=b$. Since all the elements from $A$ belong to different orbits, we have $x=y$ and equality (\ref{korre}) implies that
$$
x*^{\varepsilon_1}x_1*^{\varepsilon_2}\dots *^{\varepsilon_n}x_n*^{-\eta_k}y_k*^{-\eta_{k-1}}\dots*^{-\eta_1}y_1=x.
$$
Acting  by $\theta$ we conclude that the element $d=x_1^{\varepsilon_1}x_2^{\varepsilon_2}\dots x_n^{\varepsilon_n}y_k^{-\eta_k}\dots y_1^{-\eta_1}$ belongs to $C_{G_Q}(x)$, therefore
$$\psi(b)=(y,y_1^{\eta_1}\dots x_k^{\eta_k})=(x,dy_1^{\eta_1}\dots x_k^{\eta_k})=(x,x_1^{\varepsilon_1}\dots x_n^{\varepsilon_n})=\psi(a)$$
and the map $\psi$ is well defined.

Now we need to prove that $\psi$ is bijective. If $a=(x,w)$ for $x\in \theta(A)$, $w\in G_Q$, then since $Q$ generates $G_Q$, we have $w=x_1^{\varepsilon_1}\dots x_n^{\varepsilon_n}$ for some elements $x_1,\dots,x_n\in Q$, and numbers $\varepsilon_1,\dots,\varepsilon_n\in\{\pm1\}$ and $\psi(x*^{\varepsilon_1}x_1*^{\varepsilon_2}\dots *^{\varepsilon_n}x_n)=a$, therefore $\psi$ is surjective. For injectivity, let
\begin{align}
\notag a&=x*^{\varepsilon_1}x_1*^{\varepsilon_2}\dots *^{\varepsilon_n}x_n,\\
\notag b&=y*^{\eta_1}y_1*^{\eta_2}\dots *^{\eta_k}y_k
\end{align}
for $x,y\in A$, $x_1,\dots,x_n,y_1\dots,y_k\in Q$, $\varepsilon_1,\dots,\varepsilon_n,\eta_1,\dots,\eta_k\in\{\pm 1\}$ and
\begin{equation}\label{1inject}(x,x_1^{\varepsilon_1}\dots x_n^{\varepsilon_n})=\psi(a)=\psi(b)=(y,y_1^{\eta_2}\dots y_k^{\eta_k}).
\end{equation}
From the construction of $Q(G,A)$ follows that $x=y$ and from equality (\ref{1inject}) follows that the element $d=x_1^{\varepsilon_1}\dots x_n^{\varepsilon_n}y_k^{-\eta_k}\dots y_1^{-\eta_1}$ belongs to $C_{G_Q}(x)$. Consider the element
$$c=x*^{\varepsilon_1}x_1*^{\varepsilon_2}\dots *^{\varepsilon_n}x_n*^{-\eta_k}y_k*^{\eta_{k-1}}\dots *^{\eta_1}y_1.$$
Acting on $c$ by $\theta$ we have $\theta(c)=d^{-1}\theta(x)d=\theta(x)$, and since $\theta$ is injective, $c=x=y$. Therefore
$$b=c*^{\eta_1}y_1*^{\eta_2}\dots *^{\eta_k}y_k=x*^{\varepsilon_1}x_1*^{\varepsilon_2}\dots *^{\varepsilon_n}x_n=a$$
and $\psi$ is injective.

Now we need to prove that $\psi$ is a homomorphism. The quandle axiom (r2) says that the equality
$$(x*y)*z=(x*z)*(y*z)$$
holds for arbitrary elements $x,y,z\in Q$. Denoting by $t=x*z$ we conclude that the equality
\begin{equation}\label{hope}
t*^{-1}z*y*z=t*(y*z)
\end{equation}
holds for all $t,y,z\in Q$. Let
\begin{align}
\notag a&=x*^{\varepsilon_1}x_1*^{\varepsilon_2}\dots *^{\varepsilon_n}x_n,\\
\notag b&=y*^{\eta_1}y_1*^{\eta_2}\dots *^{\eta_k}y_k
\end{align}
be two elements from $Q$. For these elements we have $\psi(a)=(x,x_1^{\varepsilon_1}\dots x_n^{\varepsilon_n})$, $\psi(b)=(y,y_1^{\eta_1}\dots y_k^{\eta_k})$ and
\begin{equation}\label{oneside}\psi(a)*\psi(b)=(x, x_1^{\varepsilon_1}\dots x_n^{\varepsilon_n}y_k^{-\eta_k}\dots y_1^{-\eta_2}yy_1^{\eta_1}\dots y_{k-1}^{\eta_k})
\end{equation}
 Applying several times equality (\ref{hope}) we have
\begin{align}\notag a*b&=a*(y*^{\eta_1}y_1*^{\eta_2}\dots *^{\eta_{k-1}}y_{k-1})\\
\notag &=((a*^{-\eta_k}y_k)*(y*^{\eta_1}y_1*^{\eta_2}\dots *^{\eta_{k-1}}y_{k-1}))*^{\eta_k} y_k\\
\notag &\dots \\
\notag &=a*^{-\eta_k}y_k*^{-\eta_{k-1}}\dots *^{-\eta_1}y_1*y*^{\eta_1}y_1*^{\eta_2}\dots*^{\eta_k}y_k.
\end{align}
Applying $\theta$ to this equality and comparing the result with equality (\ref{oneside}) we conclude that $\psi$ is a homomorphism, and therefore $\psi$ is an isomorphism $Q\to Q(G_Q,\theta(A))$.
\end{proof}

Note that the fact that $\theta:Q\to G_Q$ is injective in Theorem \ref{quasiuniv} was used only for proving the injectivity of the map $\psi$. If we do not suppose that the map $Q\to G_Q$ is injective, then the proof of Theorem \ref{quasiuniv} guarantees the following result.
\begin{prop}\label{almostquasiuniv}
Let $Q$ be a quandle  and $A$ be a set of representatives of orbits of $Q$. Then there exists a surjective homomorphism $Q\to Q(G_Q,\theta(A))$.
\end{prop}

By Theorem \ref{quasiuniv}, if $Q$ is a quandle such that the natural map $Q\to G_Q$ is injective, then $Q=Q(G,A)$ for appropriate group $G$ and its subset $A$. Of course, the group $G$ and the set $A$ are not uniquely defined by $Q$. For example,
if $H$ is an arbitrary group, then $Q(G,A)=Q(G\times H,A)$. Moreover, if $G$ is an abelian group and $A$ is a subset of $G$ such that $|A|=n$, then independently of $A$ we have $Q(G,A)=T_n$.

Unfortunately, if the natural map $Q\to G_Q$ is not injective, then the quandle $Q$ is not necessarily a $(G,A)$-quandle. The quandles $U(1,n)$ constructed in Section~\ref{exenvexen} give the following statement.
\begin{prop} If $n\geq 2$, then $U(1,n)$ is not a $(G,A)$-quandle.
\end{prop}
\begin{proof} The quandle $U(1,n)$ has the elements $x,y_0,\dots,y_{n-1}$ with the operation
\begin{align}
\notag x*x=x*y_i=x,&& y_i*y_j=y_i,&& y_i*x=y_{i+1~({\rm mod}~n)}
\end{align}
for $i=0,\dots,n-1$. Since $x*y_0=x$ and $y_0*x=y_1\neq y_0$, from Lemma \ref{simple2} follows that  $U(1,n)$ is not a $(G,A)$-quandle.
\end{proof}
\noindent By Theorem \ref{abenvel} for $n\geq 2$ the natural map $U(1,n)\to G_{U(1,n)}$ is not injective.
\section{Commutative, latin and simple quandles are $(G,A)$-quandles}\label{latinica}
A quandle $Q$ is called commutative if $x * y = y * x$ for all $x, y \in Q$. Such quandles were studied, for example, in \cite{BaeCho}. Note that Neumann \cite{Neu} used the term ``commutative quandle'' to denote the quandle with abelian group of inner automorphisms.

 Let $Q$ be a commutative quandle. Suppose that for $x,y\in Q$ there exists $z\in Q$ such that $z*x=z*y$. Then from commutativity of $Q$ follows that $x*z=y*z$ and $x=y$. Therefore if $x\neq y$, then $S_x\neq S_y$, i.~e. the map $Q\to {\rm Inn}(Q)$ which maps an element $x$ to the automorphism $S_x$ is injective. Since ${\rm Inn}(Q)$ is the quotient of $G_Q$, the natural map $Q\to G_Q$ is also injective.

A quandle $Q$ is called latin if for every pair of elements $x,y\in Q$ there exists a unique element $z\in Q$ such that $x=y*z$. Such quandles were studied, for example, in \cite{Bon, NelTam,  Sta}. If $Q$ is a finite latin quandle then the multiplication table of $Q$ is a latin square.

Let $Q$ be a latin quandle. Suppose that for $x,y\in Q$ there exists $z\in Q$ such that $z*x=z*y$. Denote by $t=z*x$. Since $Q$ is latin, for the elements $z, t$ there exists a unique element $f$ such that $z*f=t$. The uniqueness of $f$ implies that $x=f=y$. Therefore if $x\neq y$, then $S_x\neq S_y$, i.~e. the map $Q\to {\rm Inn}(Q)$ which maps an element $x$ to the automorphism $S_x$ is injective. Since ${\rm Inn}(Q)$ is the quotient of $G_Q$, the natural map $Q\to G_Q$ is also injective.

A quandle $Q$ is called simple if every homomorphism $Q\to P$ is either injective or constant. Simple quandles were studied, for example, in \cite{Joy-1}. Every simple quandle is connected.

Let $Q$ be a simple quandle. The natural map $\theta:Q\to G_Q$ induces the homomorphism  $Q\to{\rm Conj}(G_Q)$ (which we denote by the same letter $\theta$). Suppose that this homomorphism is not injective. Then, since $Q$ is simple, $\theta$ is constant, i.~e. $\theta(x)=\theta(y)$ for all $x\in Q$ and fixed $y\in Q$. Since ${\rm Inn}(Q)$ is the quotient of $G_Q$, the map $x\mapsto S_x$ is also constant, i.~e. $S_x=S_y$ for all $x\in Q$. Therefore, for arbitrary $x,z\in Q$ we have the equality
$$x*z=S_z(x)=S_y(x)=S_x(x)=x,$$
i.~e. $Q$ is the trivial quandle. But for the trivial quandle the natural homomorphism $\theta:Q\to{\rm Conj}(G_Q)$ is injective and we have a contradiction. So, if $Q$ is a simple quandle, then the natural map $Q\to G_Q$ is always injective.

Summarizing all the above information together Theorem~\ref{quasiuniv} implies the following result.
\begin{cor}\label{simlatcom}If $Q$ is commutative, latin or simple quandle, then $Q$ is a $(G,A)$-quandle.
\end{cor}
\begin{proof}
In these cases the map $Q\to G_Q$ is injective.
\end{proof}

\section{Takasaki quandles are $(G,A)$-quandles}\label{taktaktak}
Recall that the Takasaki quandle $T(H)$ of an abelian group $H$ is the quandle $\langle H,*\rangle$ with the operation given by $x*y=2y-x$ for $x,y\in H$. In particular, if $H$ is the cyclic group of order $n$, then the Takasaki quandle $T(H)$ is called the dihedral quandle of order $n$ and is denoted by ${\rm R}_n$. In this section we prove that if $H$ is a finitely generated abelian group, then the Takasaki quandle $T(H)$ is the $(G,A)$-quandle. Moreover, we prove that if $H$ is finite, the $G$ can be chosen finite.
\begin{lemma}\label{fram} Let $H$ be a free abelian group. Then there exists a group $G$ and a subset $A$ of $G$ with $|A|=|H:2H|$ such that the Takasaki quandle $Q=T(H)$ is the $(G,A)$-quandle.
\end{lemma}
\begin{proof}Let us prove that the natural map $\theta: Q\to G_Q$ is injective. Suppose that $\theta(x)=\theta(y)$ for
some  elements $x, y$ from $Q$. Since ${\rm Inn}(Q)$ is the quotient of $G_Q$, for all $z\in Q$ we have $2x-z=z*x=z*y=2y-z$, therefore $x=y$ and $\theta$ is injective. By Theorem~\ref{quasiuniv} we have $Q=Q(G_Q,A)$, where $A$ is the set of representatives of orbits. It is easy to see that two elements $x,y\in Q$ belong to the same orbit if and only if $x-y\in 2H$, therefore $A=|H:2H|$.
\end{proof}
\begin{lemma}\label{odddih}Let $H$ be an abelian group of odd order. Then the Takasaki quandle $Q=T(H)$ is the $(G,\{x\})$-quandle for a finite group $G$ and an element $x\in G$.
\end{lemma}
\begin{proof}Let us prove that the natural map $\theta: Q\to G_Q$ is injective. Suppose that $\theta(x)=\theta(y)$ for
some  elements $x, y$ from $Q$. Since ${\rm Inn}(Q)$ is the quotient of $G_Q$, for all $z\in Q$ we have
$$2x-z=z*x=z*y=2y-z,$$
 since $H$ has no $2$-torsion, $x=y$ and $\theta$ is injective.

Since the order of $H$ is odd, for every element $x\in H$ there exists an element $y\in H$, such that $x=2y=0*y$, therefore $Q$ is connected. By Theorem \ref{quasiuniv} we have $Q=Q(G_Q,\{x\})$ for an arbitrary element $x\in Q$.

From \cite[Lemma 2.3]{GarVen} follows that $G_Q=[G_Q,G_Q]\rtimes\langle x\rangle$. Since $Q=T(H)$ is the Takasaki quandle, for every element $y\in Q$ we have $(y*x)*x=y$, therefore $x^2$ belongs to the center of $G_Q$, in particular, $\langle x^2\rangle$ is a normal subgroup of $G_Q$. Denote by $\overset{\overline{~~~~}}{}:G_Q\to G_Q/\langle x^2\rangle$ the natural homomorphism and let us prove that the map $\varphi:Q(G_Q,\{x\})\to Q(\overline{G_Q},\{\overline{x}\})$ given by $\varphi(x,y)=(\overline{x},\overline{y})$ is the isomorphism. The fact that $\varphi$ is surjective homomorphism is obvious.

Let us prove that $\varphi$ is injective. If $(\overline{x},\overline{y})=\varphi(x,y)=\varphi(x,z)=(\overline{x},\overline{z})$, then $\overline{y}=\overline{h}\overline{z}$ for some $\overline{h}$ satisfying $\overline{h}^{-1}\overline{x}\overline{h}=\overline{x}$. Going to preimages we have $y=x^{2n}hz$ for some $n$ and $h$ satisfying $h^{-1}xh=x^{2m+1}$ for some $m$. Since $h^{-1}xh$ belongs to $Q$, $x^{2m+1}$ must be equal to $x$, therefore $y=x^{2n}hz$ for $x^{2n}h\in C_{G_Q}(x)$, therefore $(x,y)=(x,z)$ and $\varphi$ is injective. So, $Q=Q(\overline{[G_Q,G_Q]}\rtimes\mathbb{Z}_2,\{\overline{x}\})$. By \cite[Lemma 2.2]{GarVen} the group $G_Q$ is the FC-group, therefore $[G_Q,G_Q]$ is finite and $\overline{[G_Q,G_Q]}\rtimes\mathbb{Z}_2$ is finite.
\end{proof}

\begin{lemma}\label{evendih} Let $n$ be a positive integer. Then there exists a finite group $G$ and two elements $x,y\in G$ such that ${\rm R}_{2n}=Q(G,\{x,y\})$.
\end{lemma}
\begin{proof} By Corollary \ref{dihinj} and Theorem \ref{quasiuniv} the quandle ${\rm R}_{2n}$ can be presented in the form ${\rm R}_{2n}=Q(G_{{\rm R}_{2n}},\theta(A))$, where $A$ is the set of representatives of orbits of ${\rm R}_{2n}$. It is clear that ${\rm R}_{2n}={\rm Orb}(x_0)\cup {\rm Orb}(x_1)$, so, $A=\{a_0,a_1\}$ and ${\rm R}_{2n}=Q(G_{{\rm R}_{2n}},\{a_0,a_1\})$.

Similarly to the proof of Corollary~\ref{dihinj} denote by $H$ the subgroup $\langle a_0^2, a_1^2\rangle$ of $G_{{\rm R}_{2n}}$. This subgroup is normal in $G_{{\rm R}_{2n}}$ and from (\ref{sesd}) we have the short exact sequence
$$1\to H\to G_{{\rm R}_{2n}}\to D_{2n}\to 1.$$
 Let $\overset{\overline{~~~~}}{}:G_{{\rm R}_{2n}}\to G_{{\rm R}_{2n}}/H=D_{2n}$ be the natural homomorphism and let us prove that the map $\varphi:Q(G_{{\rm R}_{2n}},\{a_0,a_1\})\to Q(D_{2n},\{\overline{a_0},\overline{a_1}\})$ given by $\varphi(a_{\varepsilon},x)=(\overline{a_{\varepsilon}},\overline{x})$ is the isomorphism. The fact that $\varphi$ is surjective homomorphism is obvious.

Let us prove that $\varphi$ is injective. If $(\overline{a_{\varepsilon}},\overline{x})=\varphi(a_{\varepsilon},x)=\varphi(a_{\eta},y)=(\overline{a_{\eta}},\overline{y})$, then $\varepsilon=\eta$ (it follows from the proof of Corollary~\ref{dihinj}). Then $\overline{x}=\overline{h}\overline{y}$ for some element $\overline{h}$ which satisfies $\overline{h}^{-1}\overline{a_{\varepsilon}}\overline{h}=\overline{a_{\varepsilon}}$. Going to the preimages we have $x=a_0^{2r}a_1^{2s}hy$ for some integers $r,s$ and $h$ satisfying $h^{-1}a_{\varepsilon}h=a_{\varepsilon} a_0^{2p}a_1^{2q}$ for some integers $p,q$. Since $h^{-1}a_{\varepsilon}h$ belongs to ${\rm R}_{2n}$, from Proposition~\ref{2krep} follows that the element $a_{\varepsilon} a_0^{2p}a_1^{2q}$ must be equal to $a_{\varepsilon}$, therefore $x=a_0^{2r}a_1^{2s}hy$ for $a_0^{2r}a_1^{2s}h\in C_{G_{{\rm R}_{2n}}}(a_{\varepsilon})$, therefore $(a_{\varepsilon},x)=(a_{\eta},y)$ and $\varphi$ is injective, so, ${\rm R}_{2n}=Q(D_{2n},\{\overline{a_0},\overline{a_1}\})$.
\end{proof}

Let $Q=\langle Q,*\rangle$ and $P=\langle P,\circ\rangle$ be quandles. Denote by $Q\times P$ the quandle $\langle Q\times P, \star\rangle$ with the componentwise operation $(a,x)\star (b,y)=(a*b,x\circ y)$ for $a,b\in Q$, $x,y\in P$. The quandle $Q\times P$ is called the direct product of quandles $Q$ and $P$.
\begin{prop}\label{dirQQ} Let $Q=Q(G,A)$, $P=Q(H,B)$. Then $Q\times P=Q(G\times H, A\times B)$.
\end{prop}
\begin{proof} Denote by $\varphi$ the map $\varphi:Q(G\times H, A\times B)\to Q(G, A)\times Q(H,B)$ given by
\begin{equation}\label{directprod}
((a,b),(g,h))\varphi=((a,g),(b,h))
\end{equation}
for $a\in A$, $b\in B$, $g\in G$, $h\in H$. Let us prove that $\varphi$ is correct, i.~e. if $x=y$, then $\varphi(x)=\varphi(y)$. If $((a_1,b_1),(g_1,h_1))=((a_2,b_2),(g_2,h_2))$ for $a_1,a_2\in A$, $b_1,b_2\in B$, $g_1,g_2\in G$, $h_1,h_2\in H$, then by the definition of $(G\times H,A\times B)$-quandle we have $(a_1,b_1)=(a_2,b_2)$ and  $(g_1,h_1)=(xg_2,yh_2)$ for $(x,y)\in C_{G\times H}((a_1,b_1))=C_{G}(a_1)\times C_{H}(b_1)$. Therefore $(a_1,g_1)=(a_2,g_2)$ in $Q(G,A)$ and $(b_1,h_1)=(b_2,h_2)$ in $Q(H,B)$, hence $((a_1,g_1),(b_1,h_1))=((a_2,g_2),(b_2,h_2))$ in $Q(G,A)\times Q(H,B)$ and $\varphi$ is correct.

The fact that $\varphi$ is bijective is obvious (the inverse to $\varphi$ map can be easily found from~(\ref{directprod})). The fact that $\varphi$ is the homomorphism follows from the direct calculations. Therefore $\varphi$ is the isomorphism from $Q(G\times H,A\times B)$ to $Q(G,A)\times Q(H,B)$ and the proposition is proved.
\end{proof}
Lemmas \ref{fram}, \ref{odddih}, \ref{evendih} and Proposition \ref{dirQQ} imply the following result.
\begin{theorem}\label{fintak}
Let $H$ be a finitely generated abelian group, then there exists a group $G$ and its subset $A$ such that  $|A|=|H:2H|$ and $T(H)=Q(G,A)$. If $H$ is finite, then $G$ can be chosen finite.
\end{theorem}
\begin{proof}  If $H_1$, $H_2$ are abelian groups, then $T(H_1\times H_2)=T(H_1)\times T(H_2)$. This fact follows from the fact that for $x_1,x_2\in H_1$, $y_1, y_2\in H_2$ we have the identity
$$(x_1,y_1)*(x_2,y_2)=2(x_2,y_2)-(x_1,y_1)=(2x_2-x_1,2y_2-y_1)=(x_1*x_2,y_1*y_2).$$
Since $H$ is a finitely generated abelian group, it is a direct product of cyclic groups of power prime orders and a free abelian group $H_0$. Therefore $T(H)$ is the direct product of dihedral quandles of power prime orders and the Takasaki quandle  $T(H_0)$.
The result follows directly from Lemmas \ref{fram}, \ref{odddih}, \ref{evendih} and Proposition \ref{dirQQ}.
\end{proof}

\section{Free product of quandles}\label{freeqi}
Similarly to groups and other algebraic systems, it is possible  to present quandles by generators and relations as follows.
Given  a set of letters $X$ denote by $WQ(X)$ the set of words in the alphabet $X\cup \{*, *^{-1},(,)\}$ inductively defined by the rules
\begin{enumerate}
\item $x\in WQ(X)$ for all $x\in X$,
\item if $a,b\in WQ(X)$, then $(a)*(b), (a)*^{-1}(b)\in WQ(X).$
\end{enumerate}
The set $WQ(X)$ is called the set of quandle words on $X$. Let $\{r_i,s_i~|~i\in I\}$ be some set of words from $WQ(X)$ and $R=\{r_i=s_i~|~i\in I\}$ be the set of formal equalities. Denote by $\sim$ an equivalence relation on $WQ(X)$ given by the rule:
\begin{enumerate}
\item  $x\sim x*x\sim x*^{-1}x$ for all $x\in WQ(X)$,
\item  $x\sim \left(x*y\right)*^{-1}y\sim \left(x*^{-1}y\right)*y$ for all $x,y\in WQ(X)$,
\item $\left(x*y\right)*z\sim\left(x*z\right)*(y*z)$ for all $x,y,z\in WQ(X)$,
\item $r_i\sim s_i$ if the equality $r_i=s_i$ belongs to $R$.
\end{enumerate}
The quotient $WQ(X)/{\sim}$ with well defined operation $(x,y)\mapsto x*y$ is obviously a quandle. We say that this quandle is presented by the set of generators $X$ and the set of relations $R$ and denote it by $\langle X~|~R\rangle$. It is clear that every quandle $Q$ can be presented by the set of generators $X$ and the set of relations $R$: we can take $X=Q$ and $R$ the set of all equalities $x*y=z$ in $Q$.

For two quandle $Q=\langle X~|~R\rangle$, $P=\langle Y~|~S\rangle$ with $X\cap Y=\varnothing$ denote by $Q*P$ the quandle
$$
Q * P = \langle X \cup Y~|~R \cup S\rangle.
$$
The quandle $Q*P$ is called the free product of quandles $Q, P$.
It is clear that the free product depends on quandles $Q, P$, but not  on their presentations. Indeed, if $Q=\langle X_1~|~R_1\rangle=\langle X_2~|~R_2\rangle$, $P=\langle Y_1~|~S_1\rangle=\langle Y_2~|~S_2\rangle$, then a couple of isomorphisms $\langle X_1~|~R_1\rangle\to\langle X_2~|~R_2\rangle$, $\langle Y_1~|~S_1\rangle\to\langle Y_2~|~S_2\rangle$ can be naturally extended to the isomorphism $\langle X_1\cup X_2~|~R_1\cup R_2\rangle\to\langle Y_1\cup Y_2~|~S_1\cup S_2\rangle$.
\begin{lemma}\label{envfr}Let $Q$, $P$ be quandles. Then $G_{Q*P}=G_Q*G_P$. Moreover, if the natural maps $Q\to G_Q$, $P\to G_P$ are injective, then the natural map $\theta:Q*P\to G_{Q*P}$ is injective.
\end{lemma}
\begin{proof} For the set $X$ of letters denote by $WG(X)$ the set of all group words on $X$ and let $\alpha:WQ(X)\to WG(X)$ be the map defined by: $\alpha(x)=x$ for $x\in X$, $\alpha(a*b)=\alpha(b)^{-1}\alpha(a)\alpha(b)$, $\alpha(a*^{-1}b)=\alpha(b)\alpha(a)\alpha(b)^{-1}$ for $a,b\in WQ(X)$. If $Q=\langle X~|~R\rangle$, $P=\langle Y~|~S\rangle$, then from the definition of the enveloping group it is clear that $G_Q=\langle X~|~\alpha(R)\rangle$, $G_P=\langle Y~|~\alpha(S)\rangle$. Therefore $G_{Q*P}=G_Q*G_P$.

Let $a, b\in Q*P$ be two elements such that $\theta(a)=\theta(b)$. The elements $a,b$ can be written in the forms
$a=a_1*^{\varepsilon_2}\dots *^{\varepsilon_n} a_n$, $b=b_1*^{\eta_2}\dots *^{\eta_k} b_k$, where each $a_i,b_j$ belongs to either $P$ or $Q$ and $\varepsilon_2,\dots,\varepsilon_n,\eta_2,\dots,\eta_k\in \{\pm1\}$. It is clear that $a_1,b_1$ both belong to either $P$ or $Q$ (since $\theta$ is injective on $Q$ and $P$, and there is no nontrivial element in $G_P$ conjugate to some element in $G_Q$). Without loss of generality we can suppose that $a_1,b_1\in Q$. Denoting by
$$c=a_1*^{\varepsilon_2}\dots *^{\varepsilon_n} a_n*^{-\eta_k}b_k*^{-\eta_{k-1}}\dots *^{-\eta_2} b_2$$
we see that $\theta(c)=\theta(b_1)$.  Since
$$(a_2^{\varepsilon_2}\dots a_n^{\varepsilon_n}b_k^{-\eta_k}\dots b_2^{-\eta_2})^{-1}a_1(a_2^{\varepsilon_2}\dots a_n^{\varepsilon_n}b_k^{-\eta_k}\dots b_2^{-\eta_2})=\theta(c)=\theta(b_1)=b_1,$$
 the element $a_2^{\varepsilon_2}\dots a_n^{\varepsilon_n}b_k^{-\eta_k}\dots b_2^{-\eta_2}$ belongs to $G_Q$ (two elements from $G_Q$ are conjugated in $G_P*G_Q$ if and only if they are conjugated in $G_Q$), therefore it can be written in generators $X$
 $$a_2^{\varepsilon_2}\dots a_n^{\varepsilon_n}b_k^{-\eta_k}\dots b_2^{-\eta_2}=x_1^{\xi_1}\dots x_t^{\xi_t}$$
for $x_1,\dots, x_t\in X$, $\xi_1,\dots,\xi_t\in\{\pm1\}$. Hence we have the equality
$$c=a_1*^{\varepsilon_2}\dots *^{\varepsilon_n} a_n*^{-\eta_k}b_k*^{-\eta_{k-1}}\dots *^{-\eta_2} b_2=a_1*^{\xi_1}x_1*^{\xi_2}\dots*^{\xi_t}x_t$$
which implies that $c\in Q$. Since $\theta(c)=\theta(b_1)$ and $\theta$ is injective on $Q$, we conclude that $c=b_1$ and $a=b$.
\end{proof}

Theorem \ref{quasiuniv} and Lemma \ref{envfr} imply the following result.
\begin{theorem}\label{freere} Let $Q,P$ be quandles such that the natural maps $Q\to G_Q$, $P\to G_P$ are injective, and $A,B$ be the sets of representatives of orbits of $P$, $Q$ respectively.
Then $P*Q$ is isomorphic to $Q(G_Q*G_P,A\cup B)$.
\end{theorem}
\begin{proof} Noting that the set $A\cup B$ is the maximal subset of $P*Q$ such that every two elements from it belong to different orbits, the result follows directly from Theorem~\ref{quasiuniv} and Lemma~\ref{envfr}.
\end{proof}
\begin{cor}\label{freetrivfree}The free quandle $FQ_n$ is the free product of $n$ trivial quandles $T_1$ with one element
$FQ_n = \underbrace{T_1 \ast T_1 \ast \cdots \ast T_1}_{n}$.
\end{cor}
\begin{proof} For $i=1,\dots,n$ denote by $Q_i=\{a_i\}$ the trivial quandle with one element $a_i$. From Theorem~\ref{freere} follows that
\begin{align}
\notag Q_1*Q_2*\dots*Q_n&=Q(G_{Q_1}*G_{Q_2}*\dots*G_{Q_n},\{a_1,\dots,a_n\})\\
\notag&=Q(\langle a_1\rangle*\langle a_2\rangle*\dots*\langle a_n\rangle,\{a_1,\dots,a_n\})\\
\notag&=Q(F(A),A),
\end{align}
where $A=\{a_1,\dots,a_n\}$ and $F(A)$ is the free group on $A$. Since all $Q_i$ are isomorphic to $T_1$ and $Q(F(A),A)=FQ_n$, the corollary is proved.
\end{proof}
Corollary~\ref{freetrivfree} implies the following result.
\begin{cor}The free quandle $FQ_n$ on the set $X = \{ x_1, x_2, \dots, x_n \}$ has the presentation $FQ_n=\langle X~|~\varnothing\rangle$.
\end{cor}

{\small

\medskip

\noindent
Valeriy Bardakov\\
Sobolev Institute of Mathematics, Acad. Koptyug avenue 4, 630090 Novosibirsk, Russia,\\
Novosibirsk State University, Pirogova 1, 630090 Novosibirsk, Russia,\\
Novosibirsk State Agricultural University, Dobrolyubova 160, 630039 Novosibirsk, Russia\\
bardakov@math.nsc.ru
~\\
~\\
Timur Nasybullov\\
KU Leuven KULAK, Etienne Sabbelaan 53, 8500 Kortrijk, Belgium\\
timur.nasybullov@mail.ru
}

\end{document}